\documentclass[12pt]{amsart}
\usepackage{amsmath, amsthm, amssymb, verbatim, enumerate, url, stmaryrd} 
\usepackage[margin=1 in]{geometry}
\usepackage[all]{xy}

\newcommand{\C}{\mathbb C}

\newcommand{\Ps}{\mathbb P}

\newcommand{\cA}{\mathcal A}
\newcommand{\cC}{\mathcal C}
\newcommand{\cD}{\mathcal D}
\newcommand{\cE}{\mathcal E}
\newcommand{\cF}{\mathcal F}

\newcommand{\cI}{\mathcal I}
\newcommand{\cJ}{\mathcal J}

\newcommand{\cM}{\mathcal M}
\newcommand{\cN}{\mathcal N}
\newcommand{\cO}{\mathcal O}
\newcommand{\cP}{\mathcal P}
\newcommand{\cR}{\mathcal R}
\newcommand{\cT}{\mathcal T}
\newcommand{\cV}{\mathcal V}

\newcommand{\h}{\mathfrak h}
\newcommand{\g}{\mathfrak g}

\newcommand{\m}{\mathfrak m}

\newcommand{\GK}{\mathcal{G}\mathcal{K}}

\DeclareMathOperator{\Aut}{Aut}
\DeclareMathOperator{\Assoc}{Assoc}

\DeclareMathOperator{\Der}{Der}

\DeclareMathOperator{\Hom}{Hom}
\DeclareMathOperator{\Ext}{Ext}

\DeclareMathOperator{\spec}{Spec}

\DeclareMathOperator{\Sym}{Sym}
\DeclareMathOperator{\Lie}{Lie}
\DeclareMathOperator{\QC}{QC}
\DeclareMathOperator{\Mod}{\text{-}mod}

\newcommand{\cHom}{\mathcal{H}om}
\newcommand{\cEnd}{\mathcal{E}nd}

\newcommand{\lala}{\langle \langle}
\newcommand{\rara}{\rangle \rangle}
\newcommand{\llcom}{\llbracket}
\newcommand{\rrcom}{\rrbracket}
\newcommand{\Xdr}{X_{dR}}
\newcommand{\Anh}{\widehat{A}_n}
\newcommand{\Onh}{\widehat{\mathcal{O}}_n}
\newcommand{\JNC}{\mathcal{J}^{NC}}
\newcommand{\ind}{\textbf{ind}}

\newtheorem{thm}{Theorem}[section]
\newtheorem{prop}[thm]{Proposition}
\newtheorem{cor}[thm]{Corollary}
\newtheorem{lem}[thm]{Lemma}

\theoremstyle{remark}
\newtheorem{rem}[thm]{Remark}

\theoremstyle{definition}
\newtheorem{example}[thm]{Example}

\theoremstyle{definition}
\newtheorem{defn}[thm]{Definition}

\title{Formal Geometry for Noncommutative Manifolds}
\author{Hendrik Orem}

\begin{document}

\begin{abstract}
This paper develops the tools of formal algebraic geometry in the setting of noncommutative manifolds, roughly ringed spaces locally modeled on the free associative algebra. We define a notion of noncommutative coordinate system, which is a principal bundle for an appropriate group of local coordinate changes. These bundles are shown to carry a natural flat connection with properties analogous to the classical Gelfand-Kazhdan structure.

Every noncommutative manifold has an underlying smooth variety given by abelianization. A basic question is existence and uniqueness of noncommutative thickenings of a smooth variety, i.e., finding noncommutative manifolds abelianizing to a given smooth variety. We obtain new results in this direction by showing that noncommutative coordinate systems always arise as reductions of structure group of the commutative bundle of coordinate systems on the underlying smooth variety; this also explains a relationship between $\cD$-modules on the commutative variety and sheaves of modules for the noncommutative structure sheaf.
\end{abstract}
\maketitle
\tableofcontents

\section{Introduction}

Throughout $X$ is a smooth variety over $\C$ of dimension $n$. 
\begin{defn}
   Let $A$ be an associative $\C$-algebra. Define the lower central series filtration as follows. Set $L_1(A) = A$, and $L_k(A) = [A, L_{k-1}(A)]$. Then the \emph{lower central series ideals} are 
    \[ M_k(A) = AL_k(A)A = AL_k(A). \]
    We say that $A$ is \emph{NC-complete} if $A$ is complete for the filtration $M_k(A)$. The \emph{NC-completion} of an algebra is its completion with respect to the filtration $M_k$.
\end{defn}
The remarkable paper of Kapranov \cite{k} puts forward a framework of noncommutative geometry in which the local objects are NC-complete algebras with a smoothness property.
\begin{defn}
\label{defn:NC}
Let $A$ be a NC-complete $\C$-algebra, $\pi:A \to A_{ab}$ the abelianization map, and $x \in \spec A_{ab}$. Denote by $S = \pi^{-1}(\overline{S}) \subset A$ the multiplicative subset corresponding to a multiplicative subset $\overline{S} \subset A_{ab}$. 
\begin{enumerate}[1]
\item The \emph{stalk of $A$ at $x$} is the direct limit of all localizations $A[S^{-1}]$, where $\overline{S} \subset A_{ab}$ runs over multiplicative subsets of functions not vanishing at $x$.
\item An NC-complete algebra $A$ is \emph{NC-smooth of dimension $n$} if all of its completed stalks are isomorphic to $\Anh = k\lala x_1, \dots, x_n\rara$. 
\item An \emph{affine NC-manifold} is a pair $(\spec A_{ab}, A)$ of a smooth affine variety and NC-smooth algebra $A$ abelianizing to it.
\item An \emph{NC-manifold} is a smooth variety $X$ of dimension $n$ and a sheaf $\cA$ of associative $\C$-algebras which is locally an affine NC-manifold of dimension $n$.
\item An \emph{NC-thickening} of a smooth variety $X$ is a sheaf of algebras $\cA$ such that $(X, \cA)$ is an NC-manifold. NC-thickenings of $X$ form a category $NC-Th_X$.
\end{enumerate}
\end{defn}
\begin{rem}
NC-manifolds were first defined using a different filtration by commutator ideals; however, this filtration induces the same topology as $M_k$ (see claim 2.4 of \cite{jo}), and therefore the notion NC-complete is the same.
\end{rem}

\begin{example}
Affine examples are the (NC-completions of) the \emph{locally free algebras} in \cite{jo}:
\begin{enumerate}[1]
\item The completed free algebra $\Anh$ is an NC-thickening of the formal disk.
\item For any smooth complete intersection $X=\C[x_1,\ldots,x_{n+m}]/\langle f_1,\ldots, f_m\rangle$, the algebra $A=A_{n+m}/\langle\tilde{f}_1,\ldots, \tilde{f}_m\rangle$, where each $\tilde{f}_i \in f_i + M_2$, is locally free. Its NC-completion is an NC-thickening of $X$.
\item Let $A$ be locally free, and recall the notation of definition \ref{defn:NC}. If $\bar{S}$ is any multiplicative subset of $A/M_2$ without zero divisors, then $A[S^{-1}]$ is locally free.
\item Any formally smooth (sometimes called quasi-free) algebra is locally free, by the formal tubular neighborhood theorem (\cite{cq}, Section 6, Theorem 2).
\end{enumerate}
Section 5 of \cite{k} contains non-affine examples:
\begin{enumerate}[1]
\item There is an NC-thickening of projective space $\Ps^n$ given by gluing $n + 1$ copies of the NC-completion of $\C\langle x_1, \dots, x_n \rangle$.
\item NC-thickenings of Grassmannians and flag varieties arise as natural noncommutative variants of the classical functors of points.
\item Under certain conditions, Kapranov constructs an NC-thickening of the moduli of vector bundles on a projective variety.
\end{enumerate}
\end{example}

This paper builds on ideas of \cite{k}, \cite{jo}, and \cite{pt}, which had a unifying theme of studying NC-manifolds via commutative algebraic geometry on the abelianization. Kapranov constructs an NC-thickening of arbitrary smooth affine varieties $X$, which are unique but not functorial, so this does not give the existence of NC-thickenings of non-affine $X$. One goal of this paper is to give a new geometric criterion for the existence of NC-thickenings; we accomplish this by developing an analogue of formal algebraic geometry for NC-manifolds.

The \emph{bundle of coordinate systems $\cM$} of $X$ is defined by its fiber over $x \in X$ being the space of isomorphisms between the formal neighborhood of $x$ and the abstract formal disk $\spec \Onh = \spec \C \llcom x_1, \dots, x_n \rrcom$. It is a principal bundle for the pro-algebraic group $G_n^+$ of augmented algebra automorphisms of $\Onh$; we denote by $G_n$ the full group of algebra automorphisms. Formal geometry, the study of the bundle $\cM$, goes back to \cite{gk}, and has been used to great effect for example in \cite{fbz}, \cite{bd-ch}, \cite{bk}, \cite{g}, and \cite{k-rozansky}.

The bundle of coordinate systems carries many rich structures, encapsulating information about sheaves of jets and closely related to the $L_\infty$-spaces of \cite{c-w2}. Our goal is to reformulate questions about NC-manifolds in terms of analogous objects we call \emph{noncommutative coordinate systems}. These are principal bundles for $H_n^+ = \Aut_{aug} \Anh$, augmented automorphisms of the \emph{formal noncommutative disk}, arising from NC-thickenings roughly as frame bundles, and form a category $NC-Coord_X$.

A key aspect of classical formal geometry is the \emph{Gelfand-Kazhdan structure} on $\cM$. This is a splitting of the Atiyah sequence of $\cM$, valued in $\Lie G_n$, where $G_n$ is the full group of automorphisms of $\Onh$. As explained in \cite{k-rozansky}, this splitting carries within it the $L_\infty$ structure on $\cT_X[-1]$, as well as the flat connection on all jet bundles. In section \ref{sec:GK}, we discover a \emph{noncommutative Gelfand-Kazhdan structure} on any noncommutative coordinate system $\cN$.

\newtheorem*{thm:ThCoord}{Theorem \ref{thm:ThCoord}}
\begin{thm:ThCoord}
The functor $Coord:NC-Th_X \to NC-Coord_X$ is an equivalence of categories. The inverse is given by taking flat sections of $\Assoc_\cN(\Anh)$ under the NC-Gelfand Kazhdan structure on $\cN$.
\end{thm:ThCoord}
To a certain extent this perspective is visible in section 4.4 of \cite{k}, but the foundations and implications of noncommutative formal geometry are not systematically developed.

A recent paper \cite{pt} makes several remarkable observations about NC-thickenings of a given variety $X$. The authors define a notion of NC-connection:
\begin{defn}
\label{defn:NC-Conn}
An \emph{NC-connection} is a degree one, square zero derivation $D$ of 
\[ \cA_X := \Omega^\bullet_X \otimes_{\cO_X} \widehat{T}_{\cO_X}\Omega^1_X \]
extending the de Rham differential on $\Omega^\bullet_X$ and acting on $\Omega^1_X \subset \widehat{T}_{\cO_X} \Omega^1_X$ as 
\[ D(1 \otimes \alpha) = \alpha \otimes 1 + \nabla_1(\alpha) + \nabla_2(\alpha) + \cdots  \]
where $\nabla_i(\alpha) \in \Omega^1_X \otimes_{\cO_X} T^i(\Omega^1_X)$.

A \emph{twisted NC-connection} is a sheaf of algebras $\cJ$ and a sheaf of ideals $\cI$ such that $(\cJ, \cI)$ is locally isomorphic to $(\widehat{T}_{\cO_X} (\Omega^1_X), \widehat{T}_{\cO_X}^{\geq 1} (\Omega^1_X))$, equipped with an analogue of the derivation $D$ (see definition \ref{def:ncConn}).
\end{defn}
NC-thickenings are shown to be equivalent to NC-connections:
\newtheorem*{thm:PTCoord}{Theorem 2.3.23 in \cite{pt}}
\begin{thm:PTCoord}
Given a twisted NC-connection with derivation $D$, $\ker D$ is an NC-thickening of $X$, and this induces an equivalence of categories $NC-Conn_X \stackrel{\sim}{\to} NC-Th_X$.
\end{thm:PTCoord}
Second, the authors construct a functor from the category $\cD(X)$ of $\cD$-modules on $X$ to modules over $\cA$ for any NC-thickening $\cA$ of $X$. We explain how both the notion of NC-connection and the functor from $\cD(X)$ arises naturally from the perspective of NC-coordinate systems:
\newtheorem*{thm:ConnCoord}{Theorem \ref{thm:ConnCoord}}
\begin{thm:ConnCoord}
Let $\cA$ be an NC-thickening of $X$, and $\cN$ an NC-coordinate system. Then the NC-Gelfand-Kazhdan structure on $\cN$ induces a twisted NC-connection on $\Assoc_\cN \Anh$; this construction induces an equivalence of categories $NC-Coord_X \stackrel{\sim}{\to} NC-Conn_X$.
\end{thm:ConnCoord}
This is fully analogous to the commutative situation, where an $L_\infty$-algebra structure on $\cT_X[-1]$ arises from the geometry of the bundle of coordinate systems \cite{k}; this analogy is discussed further at the end of the introduction.

Let $X_{dR}$ denote the de Rham stack of $X$ \cite{simpson}; recall that $\QC(X_{dR}) = \cD(X)$. The classical Gelfand-Kazhdan structure descends the map $X \to BG_n^+$ classifying $\cM$ to a map $X_{dR} \to BG_n$. The appearance of the de Rham stack in noncommutative formal geometry gives the following relation between $\cD(X)$ and modules over NC-thickenings $\cA$.
\newtheorem*{prop:Dmod}{Proposition \ref{prop:Dmod} (Informal)}
\begin{prop:Dmod}
The (commutative) bundle of coordinate systems $\cM$ and an NC-coordinate system $\cN$ induce a commutative diagram 
\[ 
\xymatrix{
& X_{dR} \ar_p[dl] \ar^q[dr] & \\
BH_n \ar[rr] & & BG_n 
}
\]
which naturally gives functors 
\[ \ind_q\left(\QC(BG_n)\right) \cong \cD(X) \to \ind_p\left(\Anh\Mod(\QC(BH_n))\right) \to \cA\Mod. \] 
\end{prop:Dmod}

Note that, in particular, $\cN$ is a reduction of structure group $\cM$ from $G_n^+$ to $H_n^+$. This perspective has a number of applications. For example, we give a new description of the existence of NC-thickenings of a variety $X$.
\newtheorem*{cor:reduction}{Corollary \ref{cor:reduction}}
\begin{cor:reduction}
A smooth variety $X$ admits an NC-thickening if and only if $\Assoc_\cM(G_n^+/H_n^+)$ has global sections. 
\end{cor:reduction}
Note that the quotient $G_n^+/H_n^+$ is taken as an algebraic stack, and thus $\Assoc_\cM(G_n/H_n)$ is a sheaf of stacks over $X$. The perspective of NC-coordinate systems also provides a globalization of local statements about $H_n$. In the same way that representations of $G_n$ give via formal geometry a universal construction of natural sheaves on smooth varieties, the representation theory of $H_n$ is a source for natural sheaves on NC-manifolds.

Finally, we study how NC-coordinate systems degenerate to the associated graded along the lower central series filtration and give a categorification of the main result of \cite{jo}. 
In particular, the associated graded algebra $gr_{M_k}(\Anh)$ carries an action of $G_n$. Thus $gr_{M_k}(\Anh)\Mod$ is a $G_n$-category. The main result of \cite{jo} is that if $X = \spec A_{ab}$ is affine with NC-thickening $A$, then 
\[ gr_{M_k}(A) \cong \Gamma(X, \Assoc_{\cM}(gr_{M_k}(\Anh))^\nabla), \]
where $(\--)^\nabla$ denotes the sheaf of flat sections. We now state an informal analogue of this statement for categories of modules.
\newtheorem*{thm:LCS-Cat}{Theorem \ref{thm:LCS-Cat} (Informal)}
\begin{thm:LCS-Cat}
    There is an equivalence of categories 
    \[ gr_M(A)\Mod(\QC(X)) \stackrel{\sim}{\to} \Gamma(X, \Assoc_\cM(gr_M(\Anh)\Mod)^\nabla).\]
\end{thm:LCS-Cat}

We conclude the introduction with an explanation of how this paper and \cite{pt} relate to Kapranov's seminal paper \cite{k-rozansky}, as well as how this paper relates more broadly to noncommutative geometry and Koszul duality. Among other things, Kapranov constructed a (cochain level) equivalence between $\cJ \cO_X$ (with its natural flat connection) and $\widehat{Sym}_\cO \Omega^1$, with structure maps dual to an $L_\infty$-structure on $\cT_X[-1]$. Equivalently, Kapranov constructed an isomorphism between a natural enhancement of the formal neighborhood of the diagonal in $X \times X$ and a natural enhancement of the formal neighborhood of the zero section in the tangent bundle. The $L_\infty$-structure maps on $\cT_X[-1]$ were constructed via the geometry of the bundle of coordinate systems on $X$.

Analogously, in \cite{pt} the authors recover $\cA$ (the ``structure sheaf'' of a NC-manifold) from the kernel of a collection of structure maps on $\widehat{T}_\cO \Omega^1$. This paper shows how to construct these structure maps from the bundle of NC-coordinate systems associated with $\cA$. The interpretation of these structure maps in an analogous way as a Koszul dual homotopy algebra structure will be the subject of future work.

From the perspective of Koszul duality, this paper can be viewed as the case of ``manifolds over the associative operad'' in the following philosophy: the structure maps of the Koszul dual split the Atiyah sequence of the bundle of coordinate systems. Kapranov's result explained above is the commutative version (where the Koszul dual is the $L_\infty$-algebra $\cT_X[-1]$); here we expect to interpret the splitting maps for noncommutative manifolds in terms of the Koszul dual $A_\infty$ algebra. A version of this is already visible in sections 5.3, 5.4 and 7.1 of \cite{pt}, in an analytic setting. This suggests a deep connection between Gelfand-Kazhdan structures and Koszul duality of operads.

\subsection{Acknowledgements}

The author would like to thank David Ben-Zvi and David Jordan for many helpful conversations about formal geometry and noncommutative algebra, and Alexander Polishchuk, Travis Schedler, and Michel Van den Bergh for remarks on earlier drafts of this paper. The author thanks the NSF for support via DMS 1148490 Research and Training Grant and DMS 1103525.

\section{Preliminaries}

This section collects definitions and facts from the theories of formal algebraic geometry and noncommutative manifolds which we will use.

\subsection{Definitions and Notation} Throughout $X$ is a smooth variety of dimension $n$ over $\C$.

\begin{defn}
We define the local objects of study in this paper.
\begin{enumerate}[1]
\item Let $\Anh = \C\lala x_1, \dots, x_n \rara$ denote the completed free associative algebra on $n$ generators.
\item Let $\Onh = \C\llcom x_1, \dots, x_n \rrcom$ denote the completed free commutative algebra on $n$ generators.
\item Let $H_n$ (resp. $H_n^+$) denote the group of algebra automorphisms (respectively augmented $\C$-algebra automorphisms) of $\Anh$.
\item Let $\h_n$ (resp. $\h_n^+$) denote the Lie algebra of $H_n$ (resp. $H_n^+$). 
\item Let $G_n$ (resp. $G_n^+$) denote the group of algebra automorphisms (respectively augmented $\C$-algebra automorphisms) of $\Onh$.
\item Let $\g_n$ (resp. $\g_n^+$) denote the Lie algebra of $G_n$ (resp. $G_n^+$). Concretely, this is the Lie algebra of vector fields on the formal disk (resp. those vector fields vanishing at the origin).
\end{enumerate}
\end{defn}

\begin{defn}
The \emph{de Rham stack} $X_{dR}$ associated to $X$ is given on affine schemes $S$ by $X_{dR}(S) = X(S^{red})$, where $S^{red}$ is the reduced subscheme of $S$.
\end{defn}

\begin{defn}
Let $\cF$ be a sheaf on $X$ with flat connection $\nabla$. We denote sheaf of flat sections $\left( \cF \right)^\nabla$.
\end{defn}

\begin{defn}
Let $U \subset X$ be open. Define $NC-Th_U$ to be the category whose objects are NC-manifolds $\cA$ with abelianization $\cO_U$, and morphisms given by morphisms of algebras identical on $\cO_U$.
\end{defn}

\begin{prop}[\cite{k} 4.3.1]
The assignment $U \mapsto NC-Th_U$ is a stack of groupoids locally trivial in the Zariski topology.
\end{prop}

\begin{defn}[\cite{pt}, 2.3.22]
  \label{def:ncConn}
  We define the category $NC-Conn_X$ of \emph{twisted NC-connections on $X$} as follows. Objects are tuples $(\cT, \cJ, \varphi, D)$ where $\cT$ is a sheaf of $\cO_X$-algebras, $\cJ \subset \cT$ is a sheaf of two-sided ideals Zariski locally isomorphic to the pair $\left(\widehat{T}_{\cO_X} \Omega^1_X, \widehat{T}_{\cO_X}^{\geq 1} \Omega^1_X\right)$, $\varphi$ is an isomorphism of graded algebras 
  \[
  \varphi = \left( \oplus \varphi_k \right):\bigoplus_{k \geq 0} \cJ^k/\cJ^{k+1} \stackrel{\sim}{\longrightarrow} T_{\cO_X} \Omega^1_X,
  \]
  and $D$ is a differential of degree 1 on $\Omega_X^\bullet \otimes_{\cO_X} \cT$ such that $D$ restricts to the de Rham differential on $\Omega_X^\bullet$ and the diagram 
  \[
\xymatrix{
\cJ \ar@{^{(}->}[r] \ar@/_/@{->>}[drr]& \cT \ar^-D[r] & \Omega^1_X \otimes_{\cO_X} \cT \ar@{->>}[r] & \Omega^1_X \otimes_{\cO_X} \cT/\cJ \ar^-\sim[r] &  \Omega^1_X \\
& & \cJ/\cJ^2 \ar@/_/_-{\varphi_1}^-\sim[urr] & & 
}
  \]
  commutes.
  
  A morphism between $(\cT, \cJ, \varphi, D)$ and $(\cT', \cJ', \varphi', D')$ is a morphism $\varphi:\cT \to \cT'$ of $\cO_X$-algebras sending $\cJ$ to $\cJ'$ which induces the identity on $T_{\cO_X} \Omega^1_X$ under $\varphi$ and $\varphi'$, and such that 
  \[ id \otimes f: \Omega_X^\bullet \otimes_{\cO_X} \cT \to \Omega_X^\bullet \otimes_{\cO_X} \cT' \]
  is compatible with $D$ and $D'$.
\end{defn}

One of the main results of \cite{pt}, theorem 2.3.23 in loc. cit., is the following 
\begin{thm}
  Let $(\cT, \cJ, \varphi, D)$ be a twisted NC-connection on $X$. Then $\ker D \subset \cT= \left( \Omega_X^\bullet \otimes_{\cO_X} \cT \right)^0$ is an NC-thickening of $X$, and this extends to an equivalence of categories 
  \[ NC-Conn_X \stackrel{\sim}{\to} NC-Th_X.\]
\end{thm}

\subsection{Recollections on Formal Geometry}
\label{sec:recallFG}

Our (greatly abridged) treatment has been heavily influenced by \cite{bk} and \cite{fbz}, but most closely follows \cite{g}; details can also be found in \cite{jo}. 

Every smooth variety $X$ of dimension $n$ has a canonical $G_n^+$-torsor whose fiber at $x \in X$ is the set of isomorphisms between the formal neighborhood of $x$ and $\spec \Onh$. This bundle has the structure of a scheme of infinite type over $X$ (see \cite{fbz} for the scheme structure), and we denote it by $\cM$. The bundle $\cM$ is locally trivial over $X$ in the Zariski topology.

Notice that $\cM/G_n^+ \cong X$. However, sections of $\cM$ admit an action of the larger group $G_n$; this is known classically as the \emph{Gelfand-Kazhdan structure} on $\cM$, or the \emph{Harish-Chandra connection} in the language of \cite{jo}; see section \ref{sec:GK} for further discussion. Section 3 of \cite{g} contains the following results.
\begin{prop}
  \label{prop:gaitsgoryFG}
We have an isomorphism $\cM/G_n \to X_{dR}$. Moreover, the maps $X \to BG_n^+$ and $X_{dR} \to BG_n$ classifying $\cM$ over $X$ (resp. over $X_{dR}$) fit into a Cartesian diagram 
\[
\xymatrix{
X \ar[r] \ar^p[d] & X_{dR} \ar^q[d] \\
BG_n^+ \ar[r] & BG_n
}
\]
\end{prop}
This diagram allows us to construct many natural sheaves on $X$ (resp. $X_{dR}$) from modules for $G_n^+$ (resp. $G_n$). For example, pulling back the $G_n$-module $\Onh$ gives the sheaf $\cJ \cO$ of jets of functions; other jet sheaves arise similarly. Additionally, an algebra object $A$ of $\QC(BG_n^+)$ (resp. $\QC(BG_n)$) pulls back to an algebra object $p^* A$ of $\QC(X)$ (resp. $\QC(X_{dR})$).

\section{Noncommutative Formal Geometry}

We develop analogues of fundamental results about the bundle $\cM$ and related structures recalled in section \ref{sec:recallFG} with $\Onh = k\llcom x_1, ..., x_n \rrcom$ replaced by $\Anh = k\lala x_1, ..., x_n \rara$. This leads us to the notion of a bundle of NC-coordinate systems.

\subsection{Automorphisms of the Noncommutative Disk}
\label{sec:ncDisk}

The surprisingly strong parallel between commutative and noncommutative formal geometry arises from the following local statements. 

\begin{prop}
    \label{prop:LieQuot}
We have a diagram 
\[
\xymatrix{
 H_n^+ \ar[d] \ar[r] & H_n \ar[d] \ar[r] & H_n/H_n^+ \ar^\wr[d]  \\
 G_n^+ \ar[r] & G_n \ar[r] & G_n/G_n^+ 
}
\]
inducing an isomorphism of the quotients.
\end{prop}
\begin{proof}
Algebra automorphisms automatically preserve the commutator filtration, and the identification $\Anh/M_2 \cong \Onh$ gives the vertical arrows. To see that the third vertical arrow is an isomorphism, consider the action of $H_n$ on $\Onh$. The action is transitive, and the stabilizer of the origin is precisely $H_n^+$.
\end{proof}

The following is standard. 
\begin{lem}
\label{lem:DerAnh}
Write $\Anh = \widehat{T}_k(V)$ for a vector space $V$. Then 
\begin{align*}
\h_n = \Der_k(\Anh) &\cong \widehat{T}_k(V) \otimes_k V^*, \text{ and } \\
\h_n^+ = \Der^+_k(\Anh) &\cong \widehat{T}_k^{\geq 1}(V) \otimes_k V^*.
\end{align*}
\end{lem}

\subsection{The Functor of Noncommutative Jets}

Consider the sheaf $\cO_X \boxtimes \cA$ on $X \times X$. By proposition 2.1.5, the pre-image of the complement of the ideal of the diagonal $I_\Delta \subset \cO_X \boxtimes \cO_X$ is an Ore set in $\cO_X \boxtimes \cA$. We can then restrict to open neighborhoods of the diagonal, and take a limit.

\begin{defn}
Let $(X, \cA)$ be an NC-manifold. Then $\cO_X \boxtimes \cA$ is a sheaf on $X \times X$ (see above). Let $p: \widehat{X \times X} \to X$ be the projection from the formal neighborhood of the diagonal to $X$, and $i: \widehat{X \times X} \to X \times X$ the inclusion. We define $\JNC \cA = p_* i^* \left( \cO_X \boxtimes \cA \right)$.
\end{defn}

We have the following key fact due to Kapranov \cite{k}, proposition 4.4.1.
\begin{prop}
    \label{prop:JNC}
The sheaf $\JNC \cA$ is a sheaf of $\cO_X$-algebras locally isomorphic to 
\[ \cO_X \otimes_{\underline{\C}} \underline{\Anh}, \]
and its abelianization is isomorphic to $\cJ \cO_X$.
\end{prop}
\begin{proof}
The first statement follows from the local triviality of the formal neighborhood of $X \subset X \times X$ as a bundle on $X$. The second follows from the fact that $\cA$ abelianizes to $\cO_X$.
\end{proof}

The following is clear by construction.
\begin{prop}
\label{prop:JetFibers}
The fibers of $\JNC \cA$ are canonically identified with the completed stalks of $\cA$.
\end{prop}

The following definition appears in section 4.4 of \cite{k}.
\begin{defn}
Define the category $NC-Jet_X$ with objects given by pairs $(\cJ, \psi)$ where $\cJ$ is an $\cO_X$-algebra locally isomorphic to $\cO_X \otimes_{\underline{\C}} \underline{\Anh}$ and $\psi: \cJ \to \cJ \cO_X$ is a surjection with kernel given locally by the commutator ideal. Morphisms $(\cJ, \psi) \to (\cJ', \psi')$ are given by isomorphisms $f:\cJ \to \cJ'$ such that $\psi' \circ f = \psi$.
\end{defn}

The following equivalence of categories is proposition 4.4.2 in \cite{k}. In section \ref{sec:GK}, we will construct the quasi-inverse functor.
\begin{prop}
\label{prop:ThJet}
The functor $\JNC: NC-Th_X \to NC-Jet_X$ is an equivalence of categories.
\end{prop}

\subsection{The Frame Bundle Functor}

The functor of noncommutative jets yeilds a sheaf of $\cO_X$-algebras locally isomorphic to $\cO_X \otimes_\C \Anh$. From this, give the following 
\begin{defn}
Let $\cA$ be an NC-thickening of $X$. Define $frame(\JNC \cA) = \cN$ to be the $H_n^+$-torsor of augmented $\cO_X$-algebra isomorphisms from $\JNC \cA$ to $\cO_X \otimes_{\underline{\C}} \underline{\Anh}$. We write $Coord(\cA) = frame(\JNC \cA) = \cN$ for the \emph{bundle of NC-coordinate systems associated to $\cA$}.
\end{defn}

\begin{rem}
With the above definition, it is clear that the fibers of $\cN = Coord(\cA)$ are isomorphic to the $H_n^+$-torsor of augmented $\C$-algebra isomorphisms from $\cA_x$ to $\Anh$.
\end{rem}

\begin{prop}
\label{prop:JetCoord}
We have equivalences of sheaves 
\[ frame \left( \Assoc_{\cN}(\Anh) \right) \stackrel{\sim}{\to} \cN \] 
and 
\[ \Assoc_{frame(\JNC \cA)}(\Anh) \stackrel{\sim}{\to} \JNC \cA. \]
\end{prop}
\begin{proof}
The standard argument showing the equivalence between $GL_m$-bundles and rank $m$-vector bundles applies after obvious modifications.
\end{proof}

\begin{defn}
Define the category $NC-Coord_X$ with objects given by principal $H_n^+$-bundles $\cN$ locally trivial in the Zariski topology, equipped with isomorphisms $\cN \times_{H_n^+} G_n^+ \stackrel{\sim}{\to} \cM$ and morphisms given by compatible isomorphisms of bundles over $X$.
\end{defn}

We may now strengthen the statement of proposition \ref{prop:JetCoord}
\begin{prop}
\label{prop:JetCoordCat}
The functor $frame: NC-Jet_X \to NC-Coord_X$ is an equivalence of categories. The quasi-inverse is given by $\cN \mapsto \Assoc_\cN(\Anh)$.
\end{prop}

In summary, we presently have equivalences 
\begin{equation}
\label{eq:equiv-preGK}
\xymatrix{
NC-Th_X \ar^{\JNC}[r] \ar^{Coord}@/^3pc/[rr]& NC-Jet_X \ar^{frame}@/^1pc/[r] & NC-Coord_X \ar^{\Assoc_{\--}(\Anh)}@/^1pc/[l]
}
\end{equation}
The following subsection will give us the tools required to construct a quasi-inverse to $\JNC$, and thus also to $Coord$.

\subsection{The Noncommutative Gelfand-Kazhdan Structure}
\label{sec:GK}

In usual formal geometry, we have a diagram 
\[ 
\xymatrix{
    0 \ar[r] & ad_\cM \ar[r] \ar[d] &  \cE_\cM \ar[r] \ar^-\theta[dl] &  \cT_X \ar[r] &  0 \\
             & \Assoc_\cM(\Der_k(k\llcom x_1, ..., x_n\rrcom)) \ar^-\sim[r]  &\Assoc_\cM(\g_n) & &
}
\]
where 
\[ \cE_\cM = \left(\cT_\cM\right)^{G_n^+} \cong \Der_{\cO_X-alg}\left(\cJ \cO_X, \cO_X \otimes_{\underline{\C}} \underline{\Onh} \right) .\] 
The map $\theta$ is called the \emph{Harish-Chandra connection}, or the \emph{Gelfand-Kazhdan structure}. Rather than being a connection on $\cM$ as a $G_n^+$-bundle, it is a connection on $\cM$ as a $(\g_n, G_n^+)$-module; that is, it carries an action of $\g_n$ which is integrable along the subalgebra $\g_n^+$ (see \cite{jo} for details). Extending the action from $G_n^+$ to $\g_n$ gives a flat connection on jet bundles, whose flat sections undo the functor of jets; equivalently, it induces the diagram in proposition \ref{prop:gaitsgoryFG}.

\begin{defn}
Given $\cA \in NC-Th_X$ and $\cN = Coord(\cA)$, define 
\[ \cE_\cN = \Der_{\cO_X-alg}\left(\JNC \cA, \cO_X \otimes_{\underline{\C}} \underline{\Anh}\right). \]
\end{defn}

We wish to give an analogous diagram:
\begin{equation}
\label{eq:NC-GK}
\xymatrix{
    0 \ar[r] & ad_\cN \ar[r] \ar[d] &  \cE_\cN \ar[r] \ar^-{\theta^{NC}}[dl] &  \cT_X \ar[r] &  0 \\
             & \Assoc_\cN(\Der_k(\Anh)) \ar^-\sim[r] & \Assoc_\cN(\h_n)  & &
}
\end{equation}
In the commutative case, $\theta$ is an isomorphism, and is dual to the action map of vector fields on the formal disk on functions on the formal disk. 

\begin{prop}
\label{prop:globHn}
Let $\cN$ be an NC-coordinate system. On a trivialization of $\JNC \cA$, we have descriptions of the sheaves:
\begin{align*}
 ad_\cN = \Assoc_\cN(\h_n^+) &\cong \widehat{T}^{\geq 1}_\cO(\Omega^1_X) \otimes_{\cO_X} \cT_X  \\
 \Assoc_\cN(\h_n) &\cong \widehat{T}_\cO(\Omega^1_X) \otimes_{\cO_X} \cT_X \\
 \Assoc_\cN(\h_n/\h_n^+) & \cong \cT_X. 
\end{align*}
\end{prop}
\begin{proof}
This is a globalization of \ref{lem:DerAnh}. It follows from propositions \ref{prop:JNC} and \ref{prop:JetFibers} that the fiber of $\JNC \cA$ at $x \in X$, in turn isomorphic to the stalk $\cA_x$, is canonically identified with $\widehat{T}_k V$, where $V$ is the cotangent space to $X$ at $x$. Equivalently, the generators of $\cA_x$ abelianize to the generators of $\widehat{\cO}_{X,x}$.
\end{proof}

\begin{rem}
This is analogous to the decomposition of $ad_\cM$, where the tensor powers above are replaced with \emph{symmetric powers} of the tangent bundle; see \cite{k}, section 4.2.
\end{rem}

\begin{prop}
The sheaf of Lie algebras $\Assoc_\cN(\h_n)$ acts naturally on $\JNC \cA$, hence there is a natural map $\Assoc_\cN(\h_n) \to \cE_\cN$. This map is an isomorphism.
\end{prop}
\begin{proof}
The fibers of $\cE_\cN$ are given by $\Der_{\C-alg}\left( \cA_x, \Anh \right)$. Thus the natural action map induces an isomorphism on stalks.
\end{proof}

\begin{rem}
\label{rem:flat}
The action map defining the splitting is a map of Lie algebras; therefore the connection is flat.
\end{rem}

\begin{defn}
We call the $(\h_n, H_n^+)$-valued connection $\theta^{NC} = \alpha^{-1}$ the \emph{noncommutative Gelfand-Kazhdan structure on $\cN$}.
\end{defn}

\begin{rem}
\label{rem:connection}
The map $\theta$ induces a connection on any associated bundle $\Assoc_\cN(M)$, $M \in H_n^+\Mod$. This connection can be written 
\[ \nabla_\xi(s) = ds(\xi) + \theta(\tilde{\xi}), \]
where $\tilde{\xi}$ is any lift of $\xi \in \cT_X$ to  $\cE_\cN$.
\end{rem}

\begin{rem}
The above discussion extends the action of $H_n^+$ on $\cN \in NC-Coord_X$ to an action of $\h_n = \Lie H_n$. As in the commutative setting (see \ref{sec:recallFG}), this can also be phrased in terms of extending the action to $H_n$. We will often use the language of $H_n$ rather than $\h_n$ in order to phrase results in terms of classifying spaces.
\end{rem}

\begin{cor}
Given an NC-coordinate system $\cN$ on $X$, we have $\cN/H_n \cong X_{dR}$.
\end{cor}
\begin{proof}
This follows from proposition \ref{prop:globHn}.
\end{proof}

\begin{rem}
In the language of \cite{gr-crys}, an NC-coordinate system $\cN$ is naturally a \emph{crystal} on $X$, i.e., it is equipped with descent data from $X$ to $X_{dR}$.
\end{rem}

\section{NC-Thickenings, NC-Coordinate Systems, and NC-Connections}

This section establishes equivalences of categories between NC-thickenings of a variety $X$, bundles of NC-coordinate systems over $X$, and NC-connections on $X$.

\subsection{Relation with NC-Smooth Thickenings}

We will show an equivalence of categories between noncommutative thickenings of $X$ and noncommutative systems of coordinates on $X$.

\begin{thm}
\label{thm:ThCoord}
The functor $Coord:NC-Th_X \to NC-Coord_X$ is an equivalence of categories. The inverse is given by taking flat sections of $\Assoc_\cN(\Anh)$ under the NC-Gelfand Kazhdan structure on $\cN$.
\end{thm}
\begin{proof}
The fact that $Coord$ is an equivalence follows from propositions \ref{prop:ThJet} and \ref{prop:JetCoordCat}. The second claim follows from proposition \ref{prop:flatJet} below.
\end{proof}

\begin{prop}
\label{prop:flatJet}
There is a natural isomorphism 
\[ \cA \stackrel{\sim}{\to} \left( \JNC \cA \right)^\nabla. \]
\end{prop}
\begin{proof}
We define a map $\varphi:\cA \to \left( \JNC \cA \right)^\nabla$ as follows. A section of $\cA$ determines one of $\cA \boxtimes \cA$ via the diagonal map, hence of $\cO_X \boxtimes \cA$ via projection. For any $U \subset X$, 
\[ \left(\cO_X \boxtimes \cA\right)(U \times U) \to \JNC \cA(U). \]
We claim that the image of this composition lands in the subsheaf of flat sections of $\JNC \cA$. The formula in remark \ref{rem:connection} shows that flatness of a section is equivalent to the derivative along the base being equal to the derivative long the fiber. The above map factors through the diagonal map $\cA \to \cA \boxtimes \cA$, hence this is satisfied.

The proposition now follows from 
\begin{lem}
The stalks of $\left(\JNC \cA\right)^\nabla$ are naturally isomorphic to $\cA_x$.
\end{lem}
\begin{proof}
This can be deduced from \cite{jo}, proposition 4.23.
\end{proof}

\end{proof}

We may now add another equivalence of categories to diagram (\ref{eq:equiv-preGK}): 
\begin{equation}
\xymatrix{
NC-Th_X \ar^{\JNC}@/^1pc/[r] \ar^{Coord}@/^3pc/[rr]& \ar^{\left( \-- \right)^\nabla}@/^1pc/[l] NC-Jet_X \ar^{frame}@/^1pc/[r] & NC-Coord_X \ar^{\Assoc_{\--}(\Anh)}@/^1pc/[l] 
}
\label{eq:equiv-GK}
\end{equation}

\subsection{Relation with NC-Connections}

The aim of this section is to interpret the connection $\nabla$ in terms of the NC-connections of \cite{pt}. An equivalence of categories follows formally from the preceding section and the main theorem of \cite{pt}, but we wish to illustrate directly how NC-connections relate to NC-coordinate systems.

The noncommutative Gelfand-Kazhdan structure on $\cN$ induces a flat connection on $\JNC \cA$ (see remarks \ref{rem:flat} and \ref{rem:connection}). This flat connection can be equivalently described by 
\[ \omega \in \Omega^1_X \otimes_{\cO_X} \cEnd(\JNC \cA). \]
Note that $\cEnd(\JNC \cA)$ is locally identified with
\[ \prod_k \cHom(\cT_X^{\otimes k}, \cT_X) \cong \prod_k \cHom\left(\Omega^1_X, \left( \Omega^1_X \right)^{\otimes k} \right). \]
Thus locally $\omega$ yields an \emph{untwisted} NC-connection (see definition \ref{defn:NC-Conn}) as follows.

A flat connection $\omega \in \Omega^1 \otimes \cEnd(\cF)$ on a sheaf $\cF$ defines a differential $D$ on $\Omega^\bullet_X \otimes_{\cO_X} \cF$ extending the de Rham differential. For $\cF = \widehat{T}_{\cO_X} \Omega^1_X$, the sheaf of endomorphisms decomposes into a product as above, and so does the differential $D$. Restricting $D$ to 
\[ \Omega^1_X \subset \widehat{T}_{\cO_X} \Omega^1_X  \subset \Omega^\bullet_X \otimes_{\cO_X} \widehat{T}_{\cO_X}\Omega^1_X \]
gives an operator of the form 
\[ \sum_k \nabla_k, \quad \nabla_k: \Omega^1_X \to \Omega^1_X \otimes_{\cO_X} T^k\left( \Omega^1_X \right), \]
where the first tensor factor of $\Omega^1$ on the right-hand side above is in graded degree 1. 

The above discussion works locally in order to identify $\JNC \cA$ with $\widehat{T}_{\cO_X} \Omega^1_X$. Without this the flat connection on $\JNC \cA$ still induces a degree 1 derivation $D$ on $\Omega^\bullet_X \otimes_{\cO_X} \JNC \cA$, which satisfies $D^2 = 0$ due to the flatness of the connection. We summarize the discussion in the following 

\begin{thm}
\label{thm:ConnCoord}
The preceding construction of $D$ endows $\Assoc_\cN(\Anh) = \JNC \cA$ with the structure of a twisted NC-connection. This induces an equivalence of categories 
\[ \GK:NC-Coord_X \stackrel{\sim}{\to} NC-Conn_X \]
and the following diagram of equivalences commutes.
\[
\xymatrix{
NC-Th_X \ar^{\JNC}@/^1pc/[r] \ar^{Coord}@/^3pc/[rr]& \ar^{\left( \-- \right)^\nabla}@/^1pc/[l] NC-Jet_X \ar^{frame}@/^1pc/[r] & NC-Coord_X \ar^{\Assoc_{\--}(\Anh)}@/^1pc/[l] \ar^-{\GK}@/^1.5pc/[dl]\\
& NC-Conn_X \ar^-{\ker D}@/^1.5pc/[ul]& 
}
\]
\end{thm}
\begin{proof}
The sheaf of algebras $\JNC \cA$ has a natural ideal sheaf locally identified with $\widehat{T}^{\geq 1}_{\cO_X} \Omega^1_X$. The preceding discussion gives a construction of $D$. Recalling definition \ref{def:ncConn}, it remains to construct an isomorphism of graded algebras 
\[ \varphi: \bigoplus_{k \geq 0} \left(\JNC \cA\right)^k/\left( \JNC \cA \right)^{k+1} \stackrel{\sim}{\to} T_{\cO_X} \Omega^1_X. \]
The map of $H_n$-modules $\Anh \to \Onh$ induces a map $\JNC \cA \to \cJ \cO_X$, as required in the definition of an object of $NC-Jet_X$. The filtered components of this map give the desired maps $\varphi_i$.
\end{proof}

\section{NC-Thickenings as Reductions of Structure Group}
\label{sec:reduction}

This section gives a new criterion for the existence of NC-thickenings of a given variety $X$, and studies natural obstructions.

\subsection{Existence Results} We establish a new criterion of existence of NC-thickenings of $X$.

\begin{thm}
\label{thm:reduction}
Every $\cN \in NC-Coord_X$ arises as a reduction of structure group of $\cM$ from $G_n^+$ to $H_n^+$, and this reduction respects the natural action of $H_n$ on $\cN$ and $G_n$ on $\cM$. Equivalently, NC-coordinate systems correspond to commutative diagrams 
\[
\xymatrix{
    & X_{dR} \ar[dl] \ar^q[dr] & \\
    BH_n \ar[rr] & & BG_n
}
\]
where $q$ classifies $\cM$.
\end{thm}
\begin{proof}
The statement that $\cN$ is a reduction of $\cM$ is immediate from the definition of NC-coordinate system. This reduction intertwines the $H_n$ and $G_n$ actions by the commutativity of the diagram
\[
\xymatrix{
\cE_\cN \ar[r] \ar[d] & \cE_\cM \ar[d] \\
\Assoc_\cN(\h_n) \ar[r] & \Assoc_\cM(\g_n) 
}
\]
where the vertical arrows are the NC and classical Gelfand-Kazhdan structures.
\end{proof}

\begin{cor}
\label{cor:reduction}
A smooth variety $X$ admits an NC-thickening if and only if $\Assoc_\cM(G_n^+/H_n^+)$ has global sections. 
\end{cor}

\begin{rem}
Note that the quotient $G_n^+/H_n^+$ is taken stack-theoretically; thus the associated bundle above gives a canonically defined sheaf of stacks on all smooth varieties capturing the existence of NC-thickenings.
\end{rem}

\begin{rem}
This corollary gives a simple geometric reason for the existence of NC-thickenings of affine schemes: $\Assoc_\cM(G_n^+/H_n^+)$ will always have global sections.
\end{rem}

\subsection{Obstruction Classes and the NC-Atiyah Class}

The notion of \emph{noncommutative Atiyah class} $\alpha_\cA \in H^1(X, \Hom(\cT_X^{\otimes 2}, \cT_X))$ associated to an NC-thickening $\cA$ is due to Kapranov \cite{k}.

\begin{defn}
The \emph{sheaf of noncommutative (or NC) 1-jets $\JNC_1 \cA$} is the quotient of $\JNC \cA$ by the pre-image of higher-order jets in $\cJ \cO$.
\end{defn}

\begin{defn}
The \emph{NC-Atiyah class $\alpha_\cA$ of an NC-thickening $\cA$} is the class of 
\[ 0 \to \cT_X \to \JNC_1 \cA \to \cT_X^{\otimes 2} \to 0 \]
in $\Ext^1_X(\cT_X^{\otimes 2}, \cT_X)$. Equivalently, $\alpha_\cA$ is the class obtained by locally identifying $\JNC_1 \cA$ with $T^{\leq 2}_{\cO_X} \Omega^1_X$. 
%
\end{defn}

An equivalent way of specifying of the map $\theta$ of section \ref{sec:GK} is to give a 1-form $\omega \in \Omega^1_{\cN/X} \otimes p^*\Assoc_\cN(\h_n)$, where $p:\cN \to X$. Using the decomposition by degree of $\cN$, we can form the affine bundle $\cN^{(2)}$ whose fibers are isomorphisms $\cA_x/\m^3 \to \Anh/\m^3$ which are the identity in degree 1 (cf. section 4.2 of \cite{k-rozansky}); the form $\omega$ then yields $\omega_2 \in \Omega^1_{\cN^{(2)}/X} \otimes p^*\cHom(\cT_X^{\otimes 2}, \cT_X)$.

By \cite{k}, lemma 4.1.1, the form $\omega_2$ gives a cohomology class 
\[ [\omega_2] \in H^1(X, \cHom(\cT^{\otimes 2}_X, \cT_X)) \cong \Ext^1(\cT_X^{\otimes 2}, \cT_X). \]
Comparing definitions, we see that $\JNC_1 \cA$ is globally trivial when pulled back to $\cN^{(2)}$, and $\omega_2$ classifies the action of invariant vector fields on $\cN^{(2)}$ on this trivial sheaf $\cO_{\cN^{(2)}} \otimes \underline{\Anh/\m^3}$. Since this action gives descent data for $\JNC_1 \cA$, we must have the following.
%

\begin{prop}
The cohomology classes $[\omega_2] = \alpha_\cA$ are equal.
\end{prop}

\begin{rem}
This is fully analogous to the commutative situation: in \cite{k-rozansky}, example 4.2.2, it is observed that the usual Atiyah class of $\cT_X$ arises as the ``degree 2'' part of the commutative Gelfand-Kazhdan structure.
\end{rem}

\begin{rem}
Theorem 4.5.3 in \cite{k} says roughly that 1-smooth thickenings extend to 2-smooth thickenings precisely when the second $A_\infty$-compatibility condition is satisfied by the Atiyah class. See also the discussion of $A_\infty$-spaces in section 5.4 of \cite{pt}, analogous to the $L_\infty$-spaces of \cite{c-w2} and \cite{k-rozansky}. The broader relationship between coordinate systems and Koszul duality suggested by this pattern will be the topic of future work.
\end{rem}


\section{Degeneration to the Associated Graded and Lower Central Series}

This section recalls some results of the paper \cite{jo} and gives a categorification of the main theorem of that paper.

\subsection{The Lower Central Series Filtration}

We restate a stronger version of lemma 4.25 of \cite{jo}, which follows from the original argument:
\begin{prop}
The graded algebra $gr_M(\Anh)$ is an algebra object in $\QC(BG_n)$.
\end{prop}

For $A$ the NC-thickening of a smooth affine variety $X = \spec A_{ab}$, theorem 4.26 of \cite{jo} shows 
\begin{thm}
    \label{thm:LCS}
There is an isomorphism $\Gamma\left(X, \left(\Assoc_\cM(N_k(\Anh))\right)^\nabla\right) \cong N_k(A)$.
\end{thm}
Moreover, the fact that this isomorphism is canonical gives the sheaf-theoretic analogue for non-affine $X$ and NC-thickenings $\cA$.

\subsection{Degeneration}

\begin{defn}
Let $K_n = \Aut(gr_M(\Anh))$. Define the $K_n$-bundle $\widetilde{\cM}$ as the bundle of graded $\cO_X$-algebra isomorphisms from $gr_M(\cA)$ to $gr_M\left(\underline{\Anh}\right)$.
\end{defn}

Combining theorems \ref{thm:LCS} and \ref{thm:reduction}, we have the following 
\begin{cor}
    $\Assoc_\cN(K_n) \cong \widetilde{\cM}$ for any $\cN \in NC-Coord_X$. 
\end{cor}

\subsection{Categorical Geometry of Lower Central Series Algebras}

We will show that the category of modules over the associated graded algebra $gr_M(A)$ is associated to the bundle of coordinate systems from $gr_M(\Anh)$-mod. Throughout this section we consider differential graded categories of modules.

It is a general fact that if a group $G$ acts on an algebra $A$, then $A\Mod$ is a $G$-category \cite{fg}. The following result appears in section 3.2.5 of \cite{g}.
\begin{lem}
\label{lem:shvalg}
Let $A \in Alg(\QC(BG))$, and $\cP \to X$ a $G$-bundle. Denote $\cA = \Assoc_\cP(A)$. Then 
\[ \Gamma(X, \Assoc_\cP(A\Mod)) \cong \cA\Mod(\QC(X)). \]
\end{lem}
\begin{proof}
It is easy to see that $\cA\Mod(\QC(X))$ equalizes the appropriate cotensor product diagram:
\[
\xymatrix{
    Eq \ar[r] & \QC(\cP) \times (A\Mod) \ar^-{act^*}@<1ex>[r] \ar_-{triv^*}@<-1ex>[r] & \QC(\cP) \otimes \QC(G) \otimes (A\Mod) \\
              & \cA\Mod(\QC(X)) \ar@{-->}[ul] \ar[u]
}
\]
\end{proof}

We will need a slight extension of the notion of associated bundle, owing to the natural action of $G_n$ on the $G_n^+$-bundle $\cM$. Denote by $\cM \times_{G_n}(N)$ the quotient of $\cM \times N$ by the diagonal $G_n$-action. 

\begin{thm}
    \label{thm:LCS-Cat}
There is an equivalence of categories 
\[ \Gamma(X, \cM \times_{G_n} (gr_M(\Anh)\Mod)) \cong gr_M(A)\Mod. \]
\end{thm}
\begin{proof}
First recall that $\cM \times_{G_n^+} gr_M(\Anh) \cong \cJ gr_M(\Anh)$ by theorem \ref{thm:LCS}, and therefore $\cM \times_{G_n} gr_M(\Anh) \cong gr_M(\Anh)$. The theorem then follows from lemma \ref{lem:shvalg}.
\end{proof}

\begin{rem}
This can be viewed as a categorical analogue of theorem 4.26 of \cite{jo}. Intuitively, we think of $\cM \times_{G_n} (gr_M(\Anh)\Mod)$ as the flat sections of the sheaf of categories $\Assoc_\cM(gr_M(\Anh)\Mod)$.
\end{rem}

\section{Relation to $\cD$-Modules}

In section 3.2 of \cite{pt}, it is observed that there is a functor from the category of $\cD$-modules on $X$ to the category of modules over any NC-thickening $\cA$ of $X$. We wish to interpret this in terms of the appearance of the de Rham stack above.

\subsection{Construction of Functor}

Let $(\cN, \nabla)$ be an NC-coordinate system on $X$. Then recall \ref{sec:GK} that $\cN/H_n \cong X_{dR}$. Results of section \ref{sec:reduction} imply that we have a commutative diagram 
\[
\xymatrix{
    & X_{dR} \ar_p[dl] \ar^q[dr] & \\
    BH_n \ar^\pi[rr] & & BG_n
}
\]
Recall \cite{g-shvcat} that, for a map $f:X \to Y$, one has the functor 
\begin{align*}
\ind_f: \QC(Y)\Mod &\to \QC(X)\Mod \\
\cC &\mapsto \QC(X) \otimes_{\QC(Y)} \cC
\end{align*} 
so in particular $\ind_f(\QC(Y)) \cong \QC(X)$. Thus we can write 
\[ \cD(X) = \QC(X_{dR}) \cong \ind_q(\QC(BG_n)) \cong \ind_p\left( \ind_\pi(\QC(BG_n)) \right). \]
Moreover, by lemma \ref{lem:shvalg}, 
\[ \JNC \cA\Mod(\QC(X_{dR})) \cong \Gamma\left(X_{dR}, \ind_p\left(\Anh\Mod(\QC(BH_n))\right)\right). \]
Note that $\Gamma(X_{dR}, \QC(X_{dR})) \cong \QC(X_{dR})$.
We now have the following 
\begin{prop}
\label{prop:Dmod}
There are functors 
\[ \ind_\pi \QC(BG_n) \stackrel{\sim}{\to} \QC(BH_n) \to \Anh\Mod(\QC(BH_n)) \]
such that applying $\Gamma(X_{dR}, \ind_p(\--))$ induces a functor 
\begin{align*} 
\QC(X_{dR}) \cong \Gamma\left(X_{dR}, \ind_p\left( \ind_\pi \QC(BG_n) \right) \right)&\to \ind_p \left( \Anh\Mod(\QC(BH_n)) \right) \\
&\cong \JNC\cA\Mod(\QC(X_{dR})). 
\end{align*}
Moreover, the functor of flat sections $(\--)^\nabla:\QC(X_{dR}) \to Shv_\C(X)$ induces a functor 
\[ \JNC \cA\Mod(\QC(X_{dR})) \to \cA\Mod(Shv_\C(X)), \]
where $Shv_\C(X)$ denotes the category of sheaves of $\C$-vector spaces on $X$.
\end{prop}
\begin{proof}
The $H_n$-linear map $\Anh \to \C$ induces the desired functor 
\[ \C\Mod\left( \QC(BH_n) \right) \cong \QC(BH_n) \to \Anh\Mod(\QC(BH_n)).\]
It remains to see that $\cA$ acts on flat sections of a $\JNC \cA$-module. But this follows from proposition \ref{prop:flatJet} and the Leibniz rule.
\end{proof}

\begin{cor}
For every NC-thickening $\cA$ of $X$, we obtain a functor 
\[ \cD(X) \to \cA\Mod. \]
\end{cor}

\subsection{Comparison to the Functor of \cite{pt}}
The functor $\cD\Mod \to \cA\Mod$ in \cite{pt} is as follows. Given $M \in \cD\Mod$, we can write $M$ as a $\Omega^\bullet_X$-comodule via $M \otimes_{\cO_X} \Omega^\bullet_X$. Also, $\JNC \cA$ has a flat connection so we get a de Rham-comodule $\JNC \cA \otimes_{\cO_X} \Omega^\bullet_X$. We then form 
\[ K(M) = \left(M \otimes \Omega^\bullet_X\right) \otimes_{\Omega^\bullet_X} \left(\JNC \cA \otimes \Omega^\bullet_X\right). \]
This is a dg-module over $\JNC \cA \otimes \Omega^\bullet_X$. Taking the kernel of the differential in degree 0, we get that $H^0(K(M))$ is a module over $H^0(\JNC \cA \otimes \Omega^\bullet_X) \cong \cA$, as desired.

In proposition \ref{prop:Dmod}, we have the augmentation $\Anh \to \C$. This induces a functor on $H_n$-equivariant modules
\[ \QC(BH_n) \to Mod_{\Anh}(\QC(BH_n)). \]
which, associating categories to $\cN$, gives a functor 
\[ \cD(X) \to Mod_{\JNC \cA}(\cD(X)). \]
We then claim that $\cA$ acts naturally on the sheaf of flat sections of any $\JNC \cA$-module in $\cD(X)$, i.e., we have a composition 
\[ \cD(X) \to Mod_{\JNC \cA}(\cD(X)) \to Mod_\cA(Shv_\C(X)). \]
These two functors manifestly coincide.

\bibliographystyle{alpha}
\bibliography{ncc}

\end{document}